\documentclass[12pt]{article}
\usepackage{amsmath,amsfonts,amsthm,enumerate,graphicx}
\usepackage{authblk}

\newcommand{\rk}[1]{{\rm rank}(#1)}
\newcommand{\dimn}[1]{{\rm dim}(#1)}
\newcommand{\Zn}{\mathbb{Z}_n}
\newcommand{\mb}{\mathbb}

\newcommand{\PsZn}{{\cal S}(\Zn)}
\newcommand{\inpr}[2]{\langle #1,#2\rangle}
\newtheorem{defn}{{\bf Definition}}[section]

\newtheorem{lemma}[defn]{{\bf Lemma}}

\newtheorem{theo}[defn]{{\bf Theorem}}

\newtheorem{remark}[defn]{{\bf Remark}}
\newtheorem{notation}[defn]{{\bf Notation}}

\newtheorem{qn}[defn]{{\bf Question}}

\author[1,*]{L. Sunil Chandran}
\author[1]{Deepak Rajendraprasad}
\author[2]{Nitin Singh}
\affil[1] {Department of Computer Science and Automation,}
\affil[2] {Department of Mathematics,}
\affil[ ] {Indian Institute of Science, Bangalore, India.}
\affil[*] {Corresponding author: sunil@csa.iisc.ernet.in}
\title{On Additive Combinatorics of Permutations of $\mb{Z}_n$.}
\date{}

\begin{document}
\maketitle
\begin{abstract}
Let $\mb{Z}_n$ denote the ring of integers modulo $n$.  In this paper
we consider two extremal problems on permutations of $\mb{Z}_n$, namely,
the maximum size of a collection of permutations such that the sum of any two
distinct permutations in the collection is again a permutation, and the
maximum size of a collection of permutations such that the sum of any two
distinct permutations in the collection 
is not a permutation. Let the sizes be denoted by $s(n)$ and $t(n)$
respectively. The case when $n$ is even is trivial in both the cases, with
$s(n)=1$ and $t(n)=n!$. For $n$ odd, we prove $s(n)\geq (n\phi(n))/2^k$
where $k$ is the number of distinct prime divisors of $n$. When $n$ is an
odd 
prime we prove $s(n)\leq \frac{e^2}{\pi} n ((n-1)/e)^\frac{n-1}{2}$. For the
second problem, we prove $2^{(n-1)/2}.(\frac{n-1}{2})!\leq t(n)\leq 2^k.(n-1)!/\phi(n)$ when
$n$ is odd.   
\end{abstract}\smallskip


\section{Introduction}
For $n\in \mb{Z}$, let $\mb{Z}_n$ denote the ring $\{0,\ldots,n-1\}$ with
$+$ and $.$ as addition and multiplication modulo $n$ respectively. Let $\PsZn$ denote the set of all permutations
of the set $\Zn$. We are interested in obtaining bounds on the maximum
size of a subset ${\cal P}$ of $\PsZn$ in the case when two distinct permutations
in ${\cal P}$ sum up to a permutation (Section \ref{sec:sumperm}), and in
the case when two distinct permutations in ${\cal P}$ sum up to a
non-permutation (Section \ref{sec:nonperm}). As far as we know the
problems considered above are new, though a similar looking problem for
difference of permutations is well studied, in the form of mutually
orthogonal {\em orthomorphisms} of finite groups. For the sake of
completeness, we discuss the connection between difference of permutations
problem with the orthomorphisms problem in Section \ref{sec:ortho}. The
families of permutations we consider have similarities to {\em reverse
free} and {\em reverse full} families of permutations as considered by
F\"{u}redi et al. \cite{ZF} and Cibulka \cite{cibulka}. 

\section{Preliminaries}
We review some well known facts and basic results from linear algebra,
elementary number theory and combinatorics which will be used in later
sections. A comprehensive reference for linear algebra notions discussed
below is \cite{babaifrankl}. 

\begin{defn}[Vector Spaces, Bilinear Forms, Inner Products]{\rm
Let $V$ be an $n$-dimensional vector space over the field $F$, which is naturally identified
with the set $F^n$, with component-wise addition and scalar multiplication.
A {\em bilinear form} on $V$ is a function
$f: V\times V\rightarrow F$ which is linear in both arguments, i.e,
$f(\alpha u+u',\beta v+v')=\alpha\beta f(u,v)+\alpha f(u,v')+\beta
f(u',v)+f(u',v')$ for $u,u',v,v'\in V$ and $\alpha,\beta\in F$. A {\em
symmetric } bilinear form is a bilinear form $f$ with $f(u,v)=f(v,u)$ for
all $u,v\in V$. We will call a symmetric bilinear form on $V$ as an {\em
inner product} on $V$. The inner product on $V$ given by
$\inpr{u}{v}=u_1v_1+\cdots+u_nv_n$ where, 
$u=(u_1,\ldots,u_n)$ and $v=(v_1,\ldots,v_n)$ will be called the {\em
standard} inner product on $V$.
}
\end{defn}

\begin{defn}[Totally isotropic subspaces]{\rm
Let $V$ be an $n$-dimensional vector space with the standard inner product
$\langle\:,\rangle$. We say that vectors $u,v\in V$ are {\em orthogonal},
denoted by $u\perp v$, if
$\langle u,v\rangle = 0$. For a subspace $S$ of $V$ and a vector $v\in V$, we
say $v\perp S$ if $v\perp u$ for all $u\in S$. For a subspace $S$ of $V$,
we define $S^{\perp} = \{v\in V: v\perp S\}$ which is also a subspace of
$V$. We call a subspace $S$ of $V$ {\em totally isotropic} if $S\subseteq
S^{\perp}$, or in other words, any two vectors in $S$ are orthogonal.
}
\end{defn}
We know the following fact from linear algebra.
\begin{lemma}\label{lem:lafact}
Let $V$ be an $n$-dimensional vector space with the standard inner product. Then for a subspace $S$ of $V$, we have
$\dimn{S}+\dimn{S^\perp}=n$. 
\end{lemma}

Next we discuss some basic notions from elementary number theory that will
be used in the paper. An element $s\in \mb{Z}_n$ is said to be {\em
invertible} if there exists $t\in \mb{Z}_n$ such that $st=1$ (recall that
multiplication is in $\mb{Z}_n$). The set of
all invertible elements of $\mb{Z}_n$ is called the {\em unit group} of
$\mb{Z}_n$ and is denoted by $\mb{Z}_n^{\times}$. It is easily seen that
$\mb{Z}_n^{\times}$ is a group under multiplication. We know that $k\in
\mb{Z}_n$ is invertible if and only if ${\rm gcd}(k,n)=1$. The cardinality
of the set $\{k\in \mb{Z}: 1\leq k\leq n-1,\ {\rm gcd}(k,n)=1\}$ is denoted by
$\phi(n)$, also known as {\em Euler's totient function} in literature. The
following results are well known.
\begin{lemma}\label{lem:totient}
Let $n=\prod_{i=1}^k p_i^{\alpha_i}$, where $p_1,\ldots,p_k$ are distinct
prime divisors of $n$. Then $\phi(n)=\prod_{i=1}^k
p_i^{\alpha_i-1}(p_i-1)$.
\end{lemma}
\begin{lemma}[Chinese Remainder Theorem]\label{lem:crt}
Let $n=\prod_{i=1}^k p_i^{\alpha_i}$, where $p_1,\ldots,p_k$ are distinct
prime divisors of $n$. Then we have the following isomorphism,
\begin{align}\label{eq:chinese}
\mathbb{Z}_n & \longrightarrow \mathbb{Z}_{p_1^{\alpha_1}}\times
\mathbb{Z}_{p_2^{\alpha_2}}\times\cdots \times \mathbb{Z}_{p_k^{\alpha_k}}
\nonumber \\
s & \longmapsto (s \text{\rm\  mod } p_1^{\alpha_1},\ldots,s \text{\rm\ mod }
p_k^{\alpha_k}).
\end{align}
\end{lemma}
From the above lemma we see that $s=(s_1,\ldots,s_k)$ is invertible in
$\mb{Z}_n$ if and only if $s_i$ is invertible in $\mb{Z}_{p_i^{\alpha_i}}$
for all $i=1,\ldots,k$. 

Finally we note the following combinatorial result.

\begin{lemma}\label{lem:ramsey}
Suppose the edges of the complete graph $K_n$, $n\geq 2$ are colored using
$k$ colors, such that the subgraph consisting of the edges of any given color is
bipartite. Then $n\leq 2^k$.
\end{lemma}
\begin{proof}
Let $V$ denote the vertex set of the complete graph $K_n$. We show an
injection from $V$ into the set $\{0,1\}^k$. For
each $i$, choose a partition $\{A_i,B_i\}$ of $V$ such that all edges of
color $i$ are between the sets $A_i$ and $B_i$. By hypothesis, such a
bi-partition exists for all $i=1,\ldots,k$. For $v\in V$, define
$\tilde{v}=(v_1,\ldots,v_k)$ by,
\begin{align*}
v_i = \begin{cases}
		0 & \text{ if $v\in A_i$ }, \\
		1 & \text{ if $v\in B_i$ }.
\end{cases}
\end{align*}
We now show that the map $v\mapsto \tilde{v}$ is injective. Let $v,w$ be
two distinct vertices in $V$. Let $i$ be the color of the edge $vw$. Then
we see that $\tilde{v},\tilde{w}$ differ in the $i^{th}$ component. This
proves the injectivity of the map $v\mapsto \tilde{v}$ and the lemma
follows.  
\end{proof}

\begin{notation}{\rm
We will denote permutations of $\mb{Z}_n$ as $n$-tuples
$(\sigma_1,\ldots,\sigma_n)$ where
each $\sigma_i\in \mb{Z}_n$ and is distinct. This is not to be confused with
 the cycle representation of a permutation, as is customary in algebra. Since we
do not use cycle representation of permutations in this paper, we hope
there is no confusion. Let $\sigma=(\sigma_1,\ldots,\sigma_n)$
and $\tau=(\tau_1,\ldots,\tau_n)$, be $n$-tuples over $\mb{Z}_n$. Then $\sigma\pm \tau$ denotes the tuple
$(\sigma_1\pm \tau_1,\ldots,\sigma_n\pm \tau_n)$. For $c\in \mb{Z}_n$,
$c.\sigma$ will denote the tuple $(c\sigma_1,\ldots,c\sigma_n)$, and
$c+\sigma$ will denote the tuple $(c+\sigma_1,\ldots,c+\sigma_n)$.  
}
\end{notation}

\begin{lemma}\label{lem:evencase}
Let $n$ be even and $(a_1,\ldots,a_n)$ and $(b_1,\ldots,b_n)$ be two
permutations of $\mb{Z}_n$. Then $a+b$ is not a permutation of $\mb{Z}_n$.
\end{lemma}
\begin{proof}
Let $c=a+b=(c_1,\ldots,c_n)$. Treating $a,b,c$ as $n$-tuples over $\mb{Z}$
we have,  $c_i \equiv a_i+b_i$ (mod $n$). Summing up over all $i$, we have,
\begin{align*}
\sum_{i=1}^n c_i &\equiv \sum_{i=1}^n (a_i+b_i) \text{ (mod $n$) } \\
\text{ or, } \sum_{i=1}^n (i-1) &\equiv \sum_{i=1}^n 2(i-1) \text{ (mod
$n$) } \\
\text{ or, } \frac{n(n-1)}{2} &\equiv n(n-1) \equiv 0 \text{ (mod $n$) }
\end{align*}
which is a contradiction as $(n-1)/2$ is not an integer when $n$ is even.
This proves the lemma.
\end{proof}

\begin{lemma}\label{lem:distinctsum}
Let $a=(_0,\ldots,a_{n-1})$ and $b=(b_0,\ldots,b_{n-1})$
be distinct permutations of the set $\{0,\ldots,n-1\}$ such that the
component-wise sums $c_i=a_i+b_i$ are all distinct. Then there exist $0\leq
j,k\leq n-1$ such that $c_j = c_k+1$. 
\end{lemma}
\begin{proof}
Without loss of generality assume $c_i=a_i+b_i$
satisfy the ordering $c_0<c_1<\cdots < c_{n-1}$. Now suppose the claim is
not true. Then we have, 
$c_i-c_{i-1}\geq 2$ for $1\leq i\leq n-1$.  Summing up we get $2n-2 \leq \sum_{i=1}^{n-1}
(c_i-c_{i-1}) = c_{n-1}-c_0 \leq 2n-2$. Thus $c_i-c_{i-1}=2$ for all $1\leq
i\leq n-1$, which implies $c_i=2i$ for all $0\leq i\leq n-1$. Now
$a_0+b_0=c_0=0$ implies $a_0=b_0=0$. Now $c_1=a_1+b_1=2$, therefore we must
have $a_1=b_1=1$. Continuing this way, we conclude $a_i=b_i=i$ for all
$0\leq i\leq n-1$, contradicting the fact that the permutations were
distinct. 
\end{proof}

\section{Sum of Permutations being Permutation}\label{sec:sumperm}
In this section, we consider the maximum size of a collection of
permutations of $\Zn$ such that sum of any two distinct permutations in the
collection is again a permutation (not necessarily in the collection). A
collection ${\cal P}$ of permutations of $\Zn$ is said to satisfy property
$\mathbf{(P1)}$ if for any two distinct permutations $\sigma,\tau$ in
${\cal P}$, $\sigma+\tau$ is again a permutation of $\Zn$. Let $s(n)=\max\{|{\cal P}|: {\cal P}\subseteq \PsZn, {\cal P}
\text{ satisfies } \mathbf{(P1)} \}$.\medskip

From Lemma \ref{lem:evencase}, we note that $s(n)=1$ when $n$ is even. Now
we present a construction for the lower bound on $s(n)$ when $n$ is odd.
 \begin{lemma}\label{lem:lb}
Let $n$ be an odd number $\geq 3$.  Then $s(n)\geq (n\phi(n))/2^k$ where $\phi(n)$ is
the Euler's totient function and $k$ is the number of distinct prime
divisors of $n$.
\end{lemma}
\begin{proof}
Our construction is based on the following observations.
\begin{enumerate}[{\rm (a)}]
\item For a permutation $\tau$ of ${\mathbb Z}_n$, $k.\tau$ is a
permutation of ${\mathbb Z}_n$ if and only if $k$ is invertible in
${\mathbb Z}_n$.
\item For a permutation $\tau$ of ${\mathbb Z}_n$, $k+\tau$ is a
permutation for all $k\in \mathbb{Z}_n$. 
\item Let $n=\prod_{i=1}^k p_i^{\alpha_i}$, where $p_i$'s are distinct prime
divisors of $n$. Then there exists a subset $S$ of invertible elements of $\mathbb{Z}_n$ with
$|S|=\prod_{i=1}^k p^{\alpha_i-1}(p_i-1)/2$ such that for any $x,y\in S$,
$x+y$ is invertible in $\mathbb{Z}_n$. To describe the set $S$, we use the
isomorphism in Lemma \ref{lem:crt}. 
Let $S=\{s\in \mb{Z}_n: s=(s_1,\ldots,s_k) \text{ where } s_i\equiv c_i
\text{ mod } p_i \text{ for some } 1\leq c_i\leq (p_i-1)/2\}$. Note that in
the description of $S$, each $s_i$ has $p_i^{\alpha_i-1}(p_i-1)/2$ choices,
and hence the set $S$ has the desired cardinality. Further each element of $S$ is invertible in
$\mathbb{Z}_n$. Now for $s,t\in S$, we have
$s+t=(s_1+t_1,\ldots, s_k+t_k)$ where $s=(s_1,\ldots,s_k)$ and
$t=(t_1,\ldots,t_k)$. By definition of $S$, observe that
$s_i+t_i\not\equiv 0$ (mod $p_i$) for all $1\leq i\leq k$. Thus
$s+t$ is invertible in $\mathbb{Z}_n$.   
\end{enumerate}
Now consider the set 
${\cal P}=\{s.(x+0,x+1,\ldots,x+n-1): s\in S, x\in \mathbb{Z}_n\}$. By
observations (a),(b) and (c), we see that ${\cal P}$ consists of
permutations of $\mathbb{Z}_n$. Let $\sigma,\tau$ be two distinct
permutations in ${\cal P}$ with $\sigma=s.(x+0,\ldots,x+n-1)$ and
$\tau=t.(y+0,\ldots,y+n-1)$. Then
$\sigma+\tau=(sx+ty)+(s+t).(0,1,\ldots,n-1)$. Since $s+t$ is invertible, by
observations (a) and (b), we conclude that $\sigma+\tau$ is a permutation.
Thus ${\cal P}$ satisfies $\mathbf{(P1)}$. Finally we observe that $|{\cal
P}|=n.|S|=n.\prod_{i=1}^k p_i^{\alpha_i-1}(p_i-1)/2 = (n.\phi(n))/2^k$.
This proves the lemma.
\end{proof}

\begin{remark}{\rm 
We note that when $n$ is a prime number, the bound in Lemma \ref{lem:lb}
reduces to $n(n-1)/2$. 
}
\end{remark}

We now establish an upper bound on $s(n)$ in the case when $n$ is an odd
prime. 

\begin{theo}[Cauchy-Davenport]
\label{theoremCauchyDavenport}
Let $n$ be any prime number and let $A, B \subset \mb{Z}_n$. Then $|A + B| \geq \min\{n, |A| + |B| - 1 \}$, where $A + B =\{a + b : a \in A, b\in B \}$.
\end{theo}

\begin{lemma}\label{lem:ub}
Let $n$ be an odd prime. Then $s(n)\leq \frac{e^2}{\pi} n \left(\frac{n-1}{e}\right)^\frac{n-1}{2}$.
\end{lemma}
\begin{proof}
We prove the case $n=3$ seperately. It suffices to prove $s(3)\leq 3$. This
follows from the observation that among any four permutations
$\sigma_1,\sigma_2,\sigma_3,\sigma_4$ of
$\mb{Z}_3$, there exist permutations $\sigma_i,\sigma_j$ such that
$\sigma_j$ is obtained from $\sigma_i$ by interchanging exactly one pair,
and thus 
$\sigma_i+\sigma_j$ is not a permutation.

From now on we assume $n>3$. Let ${\cal P}$ with $|{\cal P}|=m$ be a collection of permutations of
${\mathbb Z}_n$ satisfying $\mathbf{(P1)}$. Let
$A^i=(a_{i1},\ldots,a_{in})$ be the permutations in ${\cal P}$ for
$i=1,\ldots,m$. Let $M$ denote the $m\times n$ matrix with $(i,j)^{th}$
entry $a_{ij}$. When $n$ is prime the ring $\mb{Z}_n$ is infact a field,
and thus we can regard the permutations $A^i$ as vectors in $n$-dimensional
vector space $(\mb{Z}_n)^n$. Let $\langle\:,\rangle$ be the standard inner product
on $({\mathbb Z}_n)^n$. Then, we note that for a permutation $A$ of
$\mathbb{Z}_n$ we have,
\begin{equation}\label{eq:innerprod}
\langle A,A\rangle = \sum_{i=0}^{n-1} i^2 = \frac{n}{6}(n-1)(2n-1) \equiv 0\quad (\text{mod } n).
\end{equation}

\noindent{\bf Claim:} $\rk{M}\leq (n-1)/2$. Let $u,v$ be vectors
corresponding to two different rows in $M$. Then from
(\ref{eq:innerprod}), we have $\langle u,u\rangle = \langle v,v\rangle=0$.
Further since $u+v$ is also a permutation, we have $\langle u+v,u+v\rangle=0$. Using these
relations, we conclude $\langle u,v\rangle = 0$ for all $u,v\in M$. Let
$S$ be the subspace spanned by the rows of $M$. Then, by earlier
observation $S$ is totally isotropic, i.e, $S\subseteq S^\perp$. 
Combining this with the relation
$\dimn{S}+\dimn{S^\perp}=n$, we have $\rk{M}=\dimn{S}\leq \lfloor
n/2\rfloor = (n-1)/2$ as $n$ is odd.

Let $r = \rk{M}$. We may assume without loss of generality, by permuting the columns of $M$ if necessary, that the first $r$ columns of $M$ are linearly independent. 
Thus there exist linear functions $f_i: (\mathbb{Z}_n)^r\rightarrow
\mathbb{Z}_n$ for $i=1,\ldots,n-r$ such that a row of $M$ can be written
as $(x_1,x_2,\ldots,x_r,y_1,y_2,\ldots,y_{n-r})$ where
$y_i=f_i(x_1,\ldots,x_r)$. Thus any two rows in $M$
have different projection on the first $r$ positions.  

\medskip
\noindent{\bf Claim:} For any odd number $k \in \{1, r\}$, let $M'$ be an
$l \times n$ submatrix of $M$ in which all the rows agree on the first $k$
positions. Then at most $(n-k)/2$ distinct elements of $\mb{Z}_n$ appear in
any column of $M'$ beyond $k$. Suppose, for the sake of contradiction, that
the $j$-th column of $M'$ for some $j > k$ contains more than $(n-k)/2$
elements of $\mb{Z}_n$. Let $B$ be the set of elements that appear in the
$j$-th column of $M'$. By Theorem~\ref{theoremCauchyDavenport}, $|B + B|
\geq  n - k + 1$. Hence $B + B$ contains some $2x$ where $x$ is the common
entry in some $i$-th column, $i \leq k$, of $M'$. Thus either $x \in B$,
which violates the property that every row  of $M'$ is a permutation or 
$\exists x', x'' \in B, \; x' \neq x''$ such that $x' + x'' = 2x$ which 
violates the property that sum of any two rows of $M'$ is a permutation.

Since $M$ is completely determined by its first $r$ rows and the above claim, the total number of rows $m$ of $M$ is at most $n \times \prod_{i=2}^{r}\left\lceil(n - i + 1)/2\right\rceil$. Hence $m \leq n \left( \prod_{j=0}^{\lfloor r/2 \rfloor -1}\left((n-1)/2 - j \right) \right)^2$ and since $r \leq (n-1)/2$, we get   $m \leq n \left( \frac{((n-1)/2)!}{(\lceil (n-1)/4 \rceil)!} \right)^2$ which by Stirling's approximation gives us $m \leq \frac{e^2}{\pi} n \left(\frac{n-1}{e}\right)^\frac{n-1}{2}$.
\end{proof}

\section{Sum of Permutations being Non-Permutation}\label{sec:nonperm}
We consider another variant of the previous problem. We now seek a
collection of permutations of $\Zn$ such that any two distinct permutations
sum up to a non-permutation. We say that a collection ${\cal P}$ of
permutations of $\Zn$ satisfies property $\mathbf{(P2)}$ if any two
distinct permutations in ${\cal P}$ sum up to a non-permutation. Let
$t(n)=\max\{|{\cal P}|: {\cal P}\subseteq \PsZn, {\cal P} \text{ satisfies
} \mathbf{(P2)}\}$. From Lemma \ref{lem:evencase}, we see that $t(n)=n!$
when $n$ is even. We consider the case when $n$ is odd in the following
lemma.

\begin{lemma}\label{lem:nonp}
Let $n$ be an odd number. Then $2^{(n-1)/2}.(\frac{n-1}{2})!\leq t(n)\leq
2^k(n-1)!/\phi(n)$, where $k$ is the number of distinct prime divisors of
$n$.
\end{lemma}
\begin{proof}
We prove $t(n)\geq 2^{(n-1)/2}.(\frac{n-1}{2})!$. We say the pair of 
permutations $(a_0,\ldots,a_{n-1})$ and $(b_0,\ldots,b_{n-1})$ of the set
$\{0,\ldots,n-1\}$ are {\em admissible} if the sums $a_i+b_i$, $0\leq i\leq
m-1$, are not all distinct 
(the addition being in $\mb{Z}$). We construct a collection ${\cal P}$ of mutually
admissible permutations of $\{0,\ldots,n-1\}$. Note that ${\cal P}$ when
viewed as a set of permutations of $\mb{Z}_n$, satisfies $\mathbf{(P2)}$.
Let $m=(n-1)/2$. For $0\leq i\leq m-1$, let $B_i=(2i,2i+1)$ and
$\overline{B}_i=(2i+1,2i)$. Let $c=c_0c_1\ldots c_{m-1}$ be a binary string
of length $m$, and $\sigma$ be a permutation of $\{0,\ldots,m-1\}$. Define,
$P_{\sigma,c} = (x_0,x_1,\ldots,x_{n-1})$ such that,
\begin{equation*}
(x_{2i},x_{2i+1})=\begin{cases}
					B_{\sigma(i)} & \text{ if } c_i=0, \\
					\overline{B}_{\sigma(i)} & \text{ if } c_1=1 
\end{cases}
\end{equation*}
for all $i=0,\ldots,m-1$. It is easily observed that $P_{\sigma,c}$ is a
permutation of the set $\{0,\ldots,n-1\}$ and $P_{\sigma,c}\neq
P_{\sigma',c'}$ for $(\sigma,c) \neq (\sigma',c')$. Let ${\cal P}$
denote the collection of permutations $\{P_{\sigma,c}\}$. We now show that
any two permutations in the collection ${\cal P}$ are admissible. Let
$P_{\sigma,c}$ and $P_{\sigma',c'}$ be two distinct permutations in ${\cal
P}$. We consider two cases: \smallskip

\noindent{{\bf Case:}} $c\neq c'$. Let $i$ be such that $c_i\neq c'_i$.
Without loss of generality assume $c_i=0, c'_i=1$. Let
$P_{\sigma,c}=(x_0,\ldots,x_{n-1})$ and
$P_{\sigma',c'}=(x'_0,\ldots,x'_{n-1})$. Then we have,
$(x_{2i},x_{2i+1})=B_{\sigma(i)} = (a,a+1)$ and
$(x'_{2i},x'_{2i+1})=\overline{B}_{\sigma'(i)} = (b+1,b)$ for some $a,b$.
Thus $x_{2i}+x'_{2i}=x_{2i+1}+x'_{2i+1}=a+b+1$, and hence the permutations
are admissible. 

\noindent{{\bf Case:}} $c=c'$. In this case we must have $\sigma\neq
\sigma'$. In the block $B_i=(2i,2i+1)$, let us call $2i$ as the {\em
little}
end and $2i+1$ as the {\em big} end. Then $c=c'$ implies that in the
component-wise addition of $P_{\sigma,c}$ and $P_{\sigma',c'}$ the little
ends are summed with little ends, and big ends are summed with big ends.
Let $L$ be the set of sums of little ends, i.e,
$L=\{2\sigma(i)+2\sigma'(i): 0\leq i\leq m-1\}$ and $B$ be the sums of big
ends, i.e., $B=\{2\sigma(i)+2\sigma'(i)+2: 0\leq i\leq m-1\}$. We show that
$L\cap B$ is non-empty which will prove that $P_{\sigma,c}$ and
$P_{\sigma',c'}$ are admissible. By Lemma \ref{lem:distinctsum}, there
exist $0\leq j,k\leq m-1$ such that $\sigma(j)+\sigma'(j) =
\sigma(k)+\sigma'(k)+1$, or
$2\sigma(j)+2\sigma'(j)=2\sigma(k)+2\sigma'(k)+2$. We see that the left
side of the identity is in $L$ and the right side is in $B$. Thus $L\cap B$
is non-empty. \smallskip

Hence the collection of permutations $P_{\sigma,c}$, when viewed as
permutations of $\mb{Z}_n$ satisfy $\mathbf{(P2)}$. Finally we note that
the number of permutations is $2^m.m!=2^{(n-1)/2}.(\frac{n-1}{2})!$. \medskip

Next we prove $t(n)\leq
2^k(n-1)!/\phi(n)$. Let $n=\prod_{i=1}^k p_i^{\alpha_i}$. We define a relation $\sim$ on the set of permutations of
$\mb{Z}_n$ by $\sigma \sim \tau$ if $\sigma = t + s.\tau$ for some $t\in
\mb{Z}_n$ and $s\in \mb{Z}_n^{\times}$ (recall that $\mb{Z}_n^{\times}$
denotes the set of invertible elements of $\mb{Z}_n$). We show that $\sim$ is an
equivalence relation on $\PsZn$. Note that $\sigma\sim \sigma$ as
$\sigma=0+1.\sigma$. To show symmetry, assume $\sigma\sim \tau$. Then
$\sigma = t+s.\tau$ for some $t\in \mb{Z}_n$ and $s\in \mb{Z}_n^\times$.
Rearranging we have, $\tau = -t + s^{-1}.\sigma$. Thus $\tau\sim\sigma$. To
show transitivity assume $\sigma\sim\tau$ and $\tau\sim \gamma$. Then we
have, $\sigma=t_1+s_1.\tau$ and $\tau=t_2+s_2.\gamma$ for some $t_1,t_2\in
\mb{Z}_n$ and $s_1,s_2\in \mb{Z}_n^{\times}$. From the previous relations
we see that $\sigma=(t_1+s_1t_2)+s_1s_2.\gamma$, and therefore
$\sigma\sim\gamma$. Thus $\sim$ is an equivalence relation on $\PsZn$. We
now make the following claim:\medskip

\noindent{\bf Claim:} Let ${\cal P}$ be a collection of permutations of
$\mb{Z}_n$ satisfying $\mathbf{(P2)}$, and $\sim$ be the equivalence
relation on $\PsZn$ as defined. Then each equivalence class of $\sim$
contains atmost $2^k$ permutations from ${\cal P}$.

Let $\sigma_1,\ldots,\sigma_m$ be permutations from ${\cal P}$ which belong
to the same equivalence class. Let $\sigma$ be a permutation such that
$\sigma_i\sim \sigma$ for all $i=1,\ldots,m$. Let $s_i,t_i$ be such that
$\sigma_i=t_i+s_i.\sigma$. Now
$\sigma_i+\sigma_j=(t_i+t_j)+(s_i+s_j).\sigma$. From observations (a) and
(b) in the proof of Lemma \ref{lem:lb}, we see that $\sigma_i+\sigma_j$ is
a non-permutation if and only if $s_i+s_j$ is not invertible in $\mb{Z}_n$.
Thus $p_r | (s_i+s_j)$ for some $r$, $1\leq r\leq k$. Consider the edge coloring
of the complete graph $K_m$ on vertex set $\{1,\ldots,m\}$ using $k$ colors
as follows. We color the edge $ij$ with the color $\min\{r\in
\{1,\ldots,k\}: p_r | (s_i+s_j)\}$. \smallskip

\noindent{\bf Subclaim:} For each color $i\in \{1,\ldots,k\}$, the edges of
color $i$ span a bipartite subgraph of $K_m$. We prove this by showing
there are no monochromatic odd cycles. Assume $v_1v_2\ldots v_l$ is a
monochromatic odd cycle. Let $i$ be the color of the edges in the cycle.
Then $p_i | (s_1+s_2), p_i | (s_2+s_3), \ldots, p_i | (s_{l-1}+s_l), p_i |
(s_l+s_1)$. Adding the above we get $p_i | 2(s_1+\cdots+s_l)$ or $p_i |
(s_1+\cdots+s_l)$ as $p_i$ is odd. Now, from the $(l-1)/2$ relations $p_i |
(s_1+s_2), p_i | (s_3+s_4), \ldots, p_i | (s_{l-2} + s_{l-1})$ we get $p_i
| (s_1+\cdots + s_{l-1})$. Combining this with $p_i | (s_1+\cdots+s_l)$ we
get $p_i | s_l$ which contradicts the fact that $s_l$ in invertible in
$\mb{Z}_n$. Thus there are no monochromatic odd cycles, and hence edges of
a fixed color span a bipartite subgraph of $K_m$.\smallskip

The subclaim allows us to apply Lemma \ref{lem:ramsey} to the coloring of
$K_m$, and hence we conclude that $m\leq 2^k$. This proves the
claim.\medskip

Finally we notice that each equivalence class of $\sim$ contains $n\phi(n)$
elements, and hence there are $n!/(n\phi(n))=(n-1)!/\phi(n)$ equivalence
classes. Now from the claim, each equivalence class contains at most $2^k$
elements from ${\cal P}$, and hence $|{\cal P}|\leq
2^k(n-1)!/\phi(n)$. 
\end{proof}

\begin{remark}{\rm
We note that for a prime number $n$, the upper bound $\sim
(n/e)^n\sqrt{n}$
whereas the lower bound $\sim (n/e)^{n/2}\sqrt{n}$. Thus there is a quadratic
gap between the upper and lower bounds. We also mention that one way to
obtain permutations summing up to a non permutation is to consider a family
of permutations with mutual reverse. Two permutations $\sigma,\tau$ of
$\{0,\ldots,n-1\}$ are said to have a {\em mutual reverse} if there exist
$i\neq j$ such that $\sigma(i)=\tau(j)$ and $\sigma(j)=\tau(i)$. We call
such a family of permutations as {\em reverse full}. Note that
a reverse full family of permutations satisfies
$\mathbf{(P2)}$. From a result of Cibulka \cite{cibulka} on size of
reverse-free families of permutations, the maximum size of a reverse full
family of permutations on $n$ symbols is at most $n^{n/2+O(\log n)}$,
though we are unaware of any non-trivial lower bound for the same. Our
construction achieves $\sim (n/e)^{n/2}\sqrt{n}$, but under a much more
flexbile constraint.
}
\end{remark}

\section{Differences of Permutations being Permutation}\label{sec:ortho}
In this section, we wish to obtain upper and lower bounds on the maximum size of
set ${\cal P}\subseteq \PsZn$ such that for any two distinct permutations
$\sigma,\tau \in {\cal P}$, $\sigma-\tau$ is also a permutation of $\Zn$.
We say that ${\cal P}\subseteq \PsZn$ satisfies property $\mathbf{(P3)}$ if
for any two $\sigma,\tau\in {\cal P}$, $\sigma\neq \tau$, $\sigma-\tau$ is
a permutation of $\Zn$. Let $f(n) = \max\{|{\cal P}|: {\cal P}\subseteq
\PsZn \text{
satisfies } \mathbf{(P3)}\}$. 
\begin{qn}{\rm
What is a good estimate for $f(n)$ as defined above ?
}
\end{qn}
The above question is only superficially different from the well studied
problem of finding mutually orthogonal orthomorphisms of finite groups (See
\cite{evan1}). For
a finite group $G$, a bijection $\theta:G\rightarrow G$ is called an {\em
orthomorphism} of $G$ if the map $x\mapsto \theta(x)-x$ is a bijection. Two
orthomorphisms $\theta,\phi$ are called {\em orthogonal} if $\theta-\phi$
is a bijection. It is not hard to see that a set of $k$ permutations of
${\mathbb Z}_n$ satisfying {\bf (P3)} gives a set of $k-1$ mutually
orthogonal orthomorphisms
of ${\mathbb Z}_n$ and vice versa. Orthogonal orthomorphisms have been used
in construction of mutually orthogonal latin squares (MOLS). 
If we write the permutations of ${\cal P}$ as rows of a matrix,
we get a $|{\cal P}|\times n$ matrix over $\Zn$ with the property that the
difference of any two distinct rows is a permutation of $\Zn$. We point out
that such matrices have been used in connection with constructions of Latin
Squares and Orthogonal Arrays (cf. \cite[Chapter 22]{vlint}). Well known
bounds for mutually orthogonal orthomorphisms yield the following lemma.
\begin{lemma}
For $n\geq 2$, we have,
\begin{enumerate}[{\rm (a)}]
\item $f(n)\leq n-1$,
\item $f(n)=1$, when $n$ is an even number,
\item $f(n)=n-1$, when $n$ is a prime number.
\end{enumerate}
\end{lemma}

\begin{qn}{\rm
Determine $f(n)$ for odd composite number $n$.
}
\end{qn}
From results in \cite{evan1} if follows that for an odd number $n>3$ and $n$ not
divisible by $9$, we have $f(n)\geq 3$.

\end{document}